\documentclass[reqno,11pt]{amsart}

\usepackage{amssymb}
\usepackage{amsgen}
\usepackage{amsmath}
\usepackage{amsthm}
\usepackage{mathrsfs}
\usepackage{cite}
\usepackage{amsfonts}

\newcommand{\J}{\mathrel{\mathscr J}} 
\newcommand{\R}{\mathrel{\mathscr R}} 
\newcommand{\eL}{\mathrel{\mathscr L}} 

\newcommand{\inv}{^{-1}}

\usepackage{xcolor}

\newtheorem{Thm}{Theorem}
\newtheorem{Prop}[Thm]{Proposition}

\newtheorem{Lemma}[Thm]{Lemma}
{\theoremstyle{definition}
}
{\theoremstyle{remark}
}

{\theoremstyle{remark}
}
{\theoremstyle{remark}
}

{\theoremstyle{remark}
}
{\theoremstyle{remark}
}
{\theoremstyle{remark}
}

\title[Size of finite irreducible semigroups of rational matrices]{A short proof of a bound on the size of finite irreducible semigroups of rational matrices}

\author{Benjamin Steinberg}
\address[B.~Steinberg]{%
    Department of Mathematics\\
    City College of New York\\
    Convent Avenue at 138th Street\\
    New York, New York 10031\\
    USA}
\email{bsteinberg@ccny.cuny.edu}

\thanks{The author was supported by the NSF grant DMS-2452324, a Simons Foundation Collaboration Grant, award number 849561, the Australian Research Council Grant DP230103184, and Marsden Fund Grant MFP-VUW2411.}
\date{\today}

\keywords{Irreducible semigroups}
\subjclass[2020]{20M30, 20M35}

\begin{document}

\maketitle
\begin{abstract}
    I give a short proof of a recent result due to Kiefer and Ryz\-hi\-kov showing that a finite irreducible semigroup of $n\times n$ matrices has cardinality at most $3^{n^2}$.
\end{abstract}

\section{Introduction}
 Kiefer and Ryzhikov~\cite{KR.size} have recently proved that an irreducible semigroup of $n\times n$ matrices over the field of rational numbers has cardinality at most $3^{n^2}$.  This bound is a major improvement over previous bounds~\cite{bumpus}, and their proof has a number of new ideas.  However, the proof can be greatly simplified by exploiting the known representation theory of finite semigroups.  I present here a short proof using ideas very close to Minkowski's proof for finite groups.    Kiefer and Ryzhikov instead reduce to Minkowski's theorem in a combinatorial \textit{tour de force}.  A key point in my proof is Rhodes's result that irreducible finite semigroups of matrices (with zero) act faithfully on the left and right of their ($0$-)minimal ideal~\cite{Rhodeschar}.

\section{The main result}
A semigroup of matrices is \emph{irreducible} if it has no nontrivial invariant subspace.
  Let $S$ be an irreducible subsemigroup of $n\times n$-matrices over $\mathbb Q$.  Without loss of generality, I  assume that $0\in S$, since adjoining it doesn't change irreducibility.  I tacitly assume $S\neq 0$, to avoid trivialities.

An ideal of a semigroup with zero is \emph{$0$-minimal} if it contains no proper nonzero ideal.
A semigroup with zero is called \emph{generalized group mapping} if it has a $0$-minimal ideal $I$ such that $S$ acts faithfully on both the left and right of $I$, in which case $I$ must be $0$-simple and the unique $0$-minimal ideal of $S$.  A semigroup $S$ with zero is \emph{$0$-simple} if $S^2\neq 0$ and $S$ has no proper nonzero ideals.  See~\cite{Arbib} or~\cite[Chapter~4]{qtheor}.   Each $0$-simple semigroup $S$ contains a maximal subgroup $G$ in $S\setminus \{0\}$ (i.e., a subsemigroup maximal for the property of being a group), and any two maximal subgroups are isomorphic~\cite[Appendix~A]{qtheor}.  

Recall that two elements of a semigroup are $\mathscr J$-equivalent (respectively, $\mathscr R$-equivalent, respectively $\mathscr L$-equivalent) if they generate the same principal two-sided  (respectively, right, respectively, left) ideal.   One puts $\mathscr H=\mathscr R\cap \mathscr L$.  Note that each maximal subgroup is an $\mathscr H$-class~\cite[Appendix~A]{qtheor}.  Finite semigroups are stable, meaning that \[sx\J s\iff sx\R s\ \text{and}\ xs\J s\iff xs\eL s.\]   See~\cite[Theorem~A.2]{qtheor}.

\begin{Lemma}[Rhodes~{\cite[Lemma~1.5]{Rhodeschar}}]\label{l:ggm}
Let $S$ be a finite irreducible semigroup containing $0$ over a field $K$.  Then $S$ is generalized group mapping.
\end{Lemma}
\begin{proof}
Let $I$ be any $0$-minimal ideal of $S$.  The $K$-span $A$ of $S$ is a simple algebra~\cite[Theorem~5.7]{CP}.  The $K$-span of $I$ is an ideal of $A$, and hence is $A$ by simplicity.  But $A$ contains the identity matrix $1$~\cite[Lemma~5.32]{CP}.   Thus $1$ is in the $K$-span of $I$, say $1=\sum_{i=1}^k c_is_i$ with $c_i\in K$, $s_i\in I$.   If $s$ and $t$ act the same on the left of $I$, then \[s=s1=\sum_{i=1}^k c_iss_i= \sum_{i=1}^k c_its_i=t1=t\] and dually for the right.  Thus $S$ is generalized group mapping.
\end{proof}

The following is a special case of Proposition 3.28 of Chapter 8 of~\cite{Arbib}.

\begin{Prop}\label{p:ggm1to1}
Let $S$ be a finite generalized group mapping semigroup with $0$-minimal ideal $I$.  Let $G$ be a maximal subgroup of $I\setminus\{0\}$.  Then a homomorphism $\varphi\colon S\to T$ is injective if and only if $\varphi(I)\cap \varphi(0)=\emptyset$ and $\varphi|_G$ is injective.
\end{Prop}
\begin{proof}
For the nontrivial direction,  assume $\varphi(s)=\varphi(s')$.  Note that $s=s'$ in $S$ if and only if $xsy=xs'y$ for all $x,y\in I$ (see Lemma~4.6.23~\cite{qtheor}), since if $s\neq s'$, there exists $x\in I$ with $xs\neq xs'$ by faithfulness of the right action, in which case by faithfulness of the left action there is $y\in I$ with $xsy\neq xs'y$.   So, let $x,y\in I$.  Since $\varphi(xsy)=\varphi(xs'y)$, the hypotheses on $\varphi$ guarantee that either $xsy=0=xs'y$ or $xsy,xs'y\in I$.  In the latter case, $xsy\mathrel{\mathscr J} x\J y\mathrel{\mathscr J} xs'y$, and hence by  stability,  $xsy\mathrel{\mathscr H} xs'y$.  But then by Green's lemma~\cite[Lemma~A.3.1]{qtheor}, there exist $a,b$ so that $r\mapsto arb$ is a bijection from the $\mathscr H$-class of $xsy,xs'y$ to $G$.  Then $\varphi(axsyb)=\varphi(axs'yb)$ implies that $axsyb=axs'yb$ since $\varphi|_G$ is injective.  But then $xsy=xs'y$, as was required.  Thus $\varphi$ is injective.
\end{proof}

Finally, we record here a well-known lemma, at least for groups.

\begin{Lemma}\label{l:overZ}
Let $S$ be a finite subsemigroup of $M_n(\mathbb Q)$.  Then $S$ is conjugate to a subsemigroup of $M_n(\mathbb Z)$.
\end{Lemma}
\begin{proof}
Let $b_1,\ldots, b_n$ be a basis for $\mathbb Q^n$.  By finiteness of $S$, the subgroup  $L=\langle \{b_1,\ldots, b_n\}\cup Sb_1\cup\cdots\cup Sb_n\rangle$ is an $S$-invariant finitely generated subgroup of $\mathbb Q^n$, and hence is free abelian.  Moreover, it has rank $n$ and spans $\mathbb Q^n$ since it contains $b_1,\ldots, b_n$.  Then $S$ is represented by integer matrices with respect to any integral basis of $L$.
\end{proof}

Now I prove the upper bound of~\cite{KR.size}.  I will use an old result of Minkowski that the kernel of the map $GL_n(\mathbb Z)\to GL_n(\mathbb Z/p\mathbb Z)$ is torsion-free for $p$ an odd prime.  For completeness, I'll sketch the proof.

\begin{Thm}[Minkowski]
Let $p$ be an odd prime.  Then $\ker (\mathrm{GL}_n(\mathbb Z)\to \mathrm{GL}_n(\mathbb Z/p\mathbb Z))$ is torsion-free.
\end{Thm}
\begin{proof}
    If it is not torsion-free, it has an element $A$ of prime order $q$.  So $A^q=1$ and $A=1+p^kC$ where $C\not\equiv 0\bmod p$ and $k\geq 1$.   Therefore, \[1=(1+p^kC)^q=1+qp^kC+\sum_{i=2}^q\binom{q}{i}p^{ik}C^i,\] and so $qC=-\sum_{i=2}^q\binom{q}{i}p^{k(i-1)}C^i\equiv 0\bmod p$.  Thus $q=p$ as $C\not\equiv 0\bmod p$.  Dividing both sides by $p$ yields $C=-\sum_{i=2}^{p-1}\binom{p}{i}p^{k(i-1)-1}C^i-p^{k(p-1)-1}C^p$.  Since $p\mid \binom{p}{i}$ for $1<i<p$, it follows $C\equiv -p^{k(p-1)-1}C^p\bmod p$, contradicting $C\not\equiv 0\bmod p$ as $p> 2$.  
\end{proof}

To improve the bound a little bit, I shall also use that the kernel has only $2$-torsion when $p=2$ (see~\cite{Serrebound}).  Recall that a semigroup is \emph{aperiodic} if all its maximal subgroups are trivial.

\begin{Thm}
Let $S\subseteq M_n(\mathbb Q)$ be irreducible.  Then $|S|\leq 3^{n^2}$.  If the maximal subgroup $G$ of the unique ($0$-)minimal ideal of $S$ has odd order (e.g., if $S$ is aperiodic), then $|S|\leq 2^{n^2}$. 
\end{Thm}
\begin{proof}
Without loss of generality assume that $S$ contains $0$ and, by Lemma~\ref{l:overZ}, that  $S\subseteq M_n(\mathbb Z)$.  By Lemma~\ref{l:ggm}, $S$ is generalized group mapping with unique $0$-minimal ideal $I$.  Let $G$ be the maximal subgroup of $I$ and  $0\neq e$ its identity. 
Let $p=2$ if $G$ has odd order, and $p=3$ if $G$ has even order.    
  Suppose that $e$ has rank $r$. Then $e\mathbb Z^n$ is a free abelian group of rank $r$, and  $\mathbb Z^n=e\mathbb Z^n\oplus (1-e)\mathbb Z^n$.  Thus by choosing an integral basis for $\mathbb Z^n$ adapted to this direct sum, we may assume without loss of generality that \[e=\begin{bmatrix} 1_r & 0_{r,n-r}\\ 0_{n-r,r} & 0_{n-r}\end{bmatrix}.\]  
Then the canonical projection $\pi\colon M_n(\mathbb Z)\to M_n(\mathbb Z/p\mathbb Z)$ satisfies $\pi(e)\neq 0$, and  thus $\pi$ separates $0$ from $I$ (as $I$ is $0$-minimal and $\pi\inv(0)\cap I$ is an ideal).  Note that $G$ then sits inside of
\[\begin{bmatrix} GL_r(\mathbb Z) & 0_{r,n-r}\\ 0_{n-r,r} & 0_{n-r}\end{bmatrix}\] and so 
 $\pi|_G$ can be identified with the restriction of the  projection $GL_r(\mathbb Z)\to GL_r(\mathbb Z/p\mathbb Z)$.  Minkowski's Theorem (or its $p=2$ variant) then implies $\pi|_G$ is injective.  Therefore, $\pi|_S$ is injective by Proposition~\ref{p:ggm1to1}, and so $|S|\leq p^{n^2}$, as required. 
%
%
%
\end{proof}

\bibliographystyle{alpha}
\bibliography{standard2}
\end{document}